\documentclass[a4paper,12pt,reqno]{amsart}
\usepackage[margin=1.3in,marginparwidth=1.5cm, marginparsep=0.5cm]{geometry}

\usepackage[T1]{fontenc} 
\usepackage[latin1]{inputenc}
\usepackage[francais,english]{babel}

\pdfoutput=1

\usepackage{amsmath,amsfonts,amssymb,amsthm,stmaryrd,bbm,graphicx,mathtools}

\usepackage{booktabs}
\usepackage{multirow}
\usepackage{siunitx}

\usepackage{microtype}

\usepackage{mathrsfs}

\usepackage{color}

\usepackage{tikz}  

\usepackage{marginnote}
\usepackage{comment}

\usepackage[colorlinks=true, pdfstartview=FitV, linkcolor=blue, citecolor=blue, urlcolor=blue,pagebackref=false]{hyperref}

\definecolor{darkgreen}{rgb}{0,0.5,0}
\definecolor{darkred}{rgb}{0.7,0,0}
\definecolor{darkblue}{rgb}{0,0,0.7}
\theoremstyle{plain}
\newtheorem{theorem}{Theorem}[section]

\newtheorem{lemma}[theorem]{Lemma}
\newtheorem{proposition}[theorem]{Proposition}

\newtheorem{condition}[theorem]{Condition}

\newtheorem{conjecture}[theorem]{Conjecture}

\theoremstyle{remark}
\newtheorem{remark}[theorem]{Remark}

\numberwithin{equation}{section}
\numberwithin{table}{section}

\newcommand{\N}{\mathbb{N}}
\newcommand{\R}{\mathbb{R}}

\newcommand{\Z}{\mathbb{Z}}

\newcommand{\Pro}{\mathbb{P}}
\newcommand{\E}{\mathbb{E}}

\newcommand{\cross}{\text{\textup{Cross}}}
\newcommand{\arm}{\text{\textup{Arm}}}

\renewcommand{\Z}{\mathbb{Z}}
\renewcommand{\N}{\mathbb{N}}

\def\calC{\mathcal{C}}

\def\calN{\mathcal{N}}

\def\Var#1{\mathrm{Var}\bigl[ #1\bigr]}

\def\Cov#1{\mathrm{Cov}\bigl[ #1\bigr]}

\def\P{\mathbb{P}} 
\def\E{\mathbb{E}} 
\def\md{\mid}

\def \eps {\varepsilon}

\def\Bb#1#2{{\def\md{\bigm| }#1\bigl[#2\bigr]}}

\def\Pb{\Bb\P}
\def\Eb{\Bb\E}

\def \p {{\partial}}

\renewcommand{\subset}{\subseteq}

\renewcommand{\hat}{\widehat}

\def\<#1{\langle #1\rangle}

\def\bi{\begin{itemize}}  
\def\ei{\end{itemize}}
\def\bnum{\begin{enumerate}} 
\def\enum{\end{enumerate}}
\def\ni{\noindent}
\def\bf{\bfseries}

\usepackage{mhequ} 

\colorlet{symbols}{blue!90!black}
\colorlet{testcolor}{green!60!black}

\def\1{\mathbf{{1}}}

\def\g2{\frac {\gamma^2} 2}

\title{Bargmann-Fock percolation is noise sensitive}

\author{Christophe Garban and Hugo Vanneuville}

\address
{Université Claude Bernard Lyon 1, CNRS UMR 5208, Institut Camille Jordan, 69622 Villeurbanne, France}
\email{garban@math.univ-lyon1.fr; vanneuville@math.univ-lyon1.fr}

\begin{document}

\maketitle


\begin{abstract}
We show that planar Bargmann-Fock percolation is noise sensitive under the Ornstein-Ulhenbeck process.  The proof is based on the {\em randomized algorithm} approach introduced by Schramm and Steif (\cite{SS}) and gives quantitative polynomial bounds on the noise sensitivity of crossing events for Bargmann-Fock. 

A rather counter-intuitive consequence is as follows. Let $F$ be a Bargmann-Fock Gaussian field in $\R^3$ and consider two horizontal planes $P_1,P_2$  at small distance $\varepsilon$ from each other.  Even though $F$ is a.s. analytic, the above noise sensitivity statement implies that the full restriction of $F$ to $P_1$ (i.e. $F_{| P_1}$) gives almost no information on the percolation configuration induced by $F_{|P_2}$.  

As an application of this noise sensitivity analysis, we provide a Schramm-Steif based proof that the near-critical window of level line percolation around $\ell_c=0$ is polynomially small. This new approach extends earlier sharp threshold results to a larger family of planar Gaussian fields. 
\end{abstract}


\section{Main results}

\subsection{Bargmann-Fock percolation is noise sensitive}\label{ss:BF}

The planar Bargmann-Fock field is the smooth centered Gaussian field $f$ on $\R^2$ defined by the following covariance kernel
\begin{align*}\label{}
\forall x,y \in \R^2, \, \Cov{f(x),f(y)}=\exp(-\frac 1 2 |x-y|^2) \, .
\end{align*}
It can be realized as the following random entire function
\begin{align}\label{e.BF}
f(x)=f(x_1,x_2)= e^{-\frac{1}{2}|x|^2} \sum_{i,j\in \N} \frac{a_{i,j}}{\sqrt{i! j!}} x_1^i x_2^j \, ,
\end{align}
where $(a_{i,j})_{i,j}$ are i.i.d standard Gaussians (the sum converges uniformly on any compact a.s.).

The percolation model induced by the level sets of smooth Gaussian fields has been studied extensively these last few years (see \cite{BG,BM,BGbis,RVa,RVb,MV,R}) and is believed to behave like Bernoulli percolation for a large family of planar Gaussian fields (see \cite{W,BSa,BDS,BSb}). In the case of the Bargmann-Fock field - and for a wide class of Gaussian fields, see Subsection \ref{ss:gen} - it is known that this percolation model exhibits a sharp phase transition at the critical level $\ell_c=0$ in the following sense. If $\ell \in \R$ and if $Q$ is a quad (i.e. a Jordan domain of $\R^2$ with piecewise smooth boundary and with two distinguished disjoint segments on $\partial Q$), we let $\cross_\ell(Q)$ be the event that there is a continuous path included in $Q \cap \{ f >-\ell \}$ that connects the two distinguished segments of $\partial Q$. Then,
\begin{itemize}
\item Theorem 1.1 of \cite{BG}: for every quad $Q$, there exists $c=c(Q) \in ]0,1[$ such that, for every $s>0$, $c \leq \Pro \left[ \cross_0(sQ) \right] \leq 1-c$;
\item Theorem 1.8 of \cite{RVb}: for every $\ell > 0$ and every quad $Q$, there exists $c=c(\ell,Q)>0$ such that, for every $s>0$, $\Pro \left[ \cross_{-\ell}(sQ) \right] \leq e^{-cs}$ and $\Pro \left[ \cross_\ell(sQ) \right] \geq 1 -e^{-cs}$.
\end{itemize}
As explained in \cite{BG,RVb}, this implies that i) if $\ell \leq 0$, a.s. there is no unbounded component in $\{ f > -\ell \}$ (this was first proved in \cite{Alex}) and ii) if $\ell > 0$, a.s. there is a unique unbounded component in $\{ f > -\ell \}$. In the present paper, we prove that the model is noise sensitive at the critical level $\ell_c=0$. Let us explain what it means. By analogy with the model of \textbf{dynamical percolation} introduced in \cite{HPS}, we consider the following dynamics we shall call \textbf{dynamical Bargmann-Fock} model:
\begin{align*}\label{}
f(t,x) = f(t,x_1,x_2)=e^{-\frac{1}{2}|x|^2} \sum_{i,j\in \N} \frac{a_{i,j}(t)}{\sqrt{i! j!}} x_1^i x_2^j \, ,
\end{align*}
where $(t \mapsto a_{i,j}(t) )_{i,j}$ are independent Ornstein-Uhlenbeck processes with invariant measure $\calN(0,1)$ i.e. $d a_{i,j}(t) = - a_{i,j}(t) dt + dB_{i,j}(t)$; $a_{i,j}(0) \sim \mathcal{N}(0,1)$. (Equivalently, $t \mapsto a_{i,j}(t)$ is a continuous centered Gaussian field on $\R_+$ with covariance $(s,t) \mapsto e^{-|t-s|}$.)
\medskip

Our main result states that Bargmann-Fock percolation is sensitive to a polynomially small noise. This result is analogous to the quantitative noise sensitivity result for Bernoulli percolation obtained by Schramm and Steif \cite{SS}.
\begin{theorem}[See Theorem \ref{thm:mainbis} below for a more general statement]\label{th.main} $ $

\ni
There exists $\alpha>0$ such that the following holds. For any quad $Q$, there exists $c=c(Q)>0$ such that, for every sequence $(t_n)_{n \geq 1}$ that satisfies $t_n \geq n^{-\alpha}$, we have
\begin{align*}\label{}
\Cov{1_{f(0) \in \cross_0(nQ)}, 1_{f(t_n)\in \cross_0(nQ)}} \leq  cn^{-\alpha} \, .
\end{align*}
\end{theorem}

\begin{remark}
Theorem \ref{th.main} (as well as consequences of this result, such as Proposition \ref{pr.cor}, or generalizations such as Theorem \ref{thm:mainbis}) also hold if we replace the event $\cross_0(nQ)$ by the event that there is a crossing by a nodal line i.e. the event that there is a continuous path included in $nQ \cap \{ f = 0 \}$ that connects the two distinguished segments of $\partial (nQ)$ (and the proof is exactly the same even though the later event is not monotone). Note that, for the Gaussian fields that we consider in this paper, it is known that this event is also non-degenerate (see \cite{BG,MV}).
\end{remark}

Thanks to the Wiener chaos expansion of the $L^2$ functional $1_{f\in \mathrm{Cross}(nQ)}$, we will extract from the above theorem the following counter-intuitive property for the 3D Bargmann-Fock field. Let $F$ be a 3D Bargman-Fock field i.e. the smooth Gaussian field on $\R^3$ with covariance
\begin{align*}\label{}
\forall x,y \in \R^3, \, \Cov{F(x),F(y)}=\exp(-\frac 1 2 |x-y|^2) \, .
\end{align*}
Note that the restriction of $F$ to any plane is a planar Bargmann-Fock field. The following proposition states that the restriction of $F$ to a horizontal plane $P$ gives almost no information about the percolation properties in another horizontal plane at small distance $\eps$ from $P$. Note that, on the contrary, knowing $F$ restricted to a slide $\R^2 \times [t,t+\varepsilon]$ (or to any open non-empty subset of $\R^3$) freezes the whole field by analytical rigidity.
\begin{proposition}\label{pr.cor}
For every $t \in \R$, let $P(t)=\R^2 \times \{t\}$. There exists $\alpha>0$ such that, for any quad $Q$, there exists $c=c(Q)>0$ such that the following holds. For every sequence $(t_n)_{n \geq 1}$ that satisfies $t_n \geq n^{-\alpha}$, we have
\begin{align*}
\Var{\Pb{F_{|P(t_n)} \in \cross_0(nQ) | F_{|P(0)}}} \leq cn^{-\alpha} \, .
\end{align*}
\end{proposition}
See Remark \ref{rem:counter_int_fourier} for another reason why the above result seems counter-intuitive (from a Fourier point of view).

\subsection{Generalization of the result and application to the phase transition}\label{ss:gen}

In this subsection, we generalize our main result to a large family of planar Gaussian fields. As in \cite{MV}, we consider a planar white noise $W$ which is the centered Gaussian field $( \int u dW)_{u \in L^2(\R^2)}$ indexed by $L^2(\R^2)$ with the following covariance kernel
\[
\E \left[ \int udW \int vdW \right] = \int uv \, .
\]
Let $q \, : \, \R^2 \rightarrow \R$ be an $L^2$ function. Most of the time, we will ask that $q$ satisfies some of the following conditions listed below.
\begin{condition}[Symmetry and regularity]\label{cond:sym}
The function $q$ is not identically zero, for every
$x \in \R^2$, $q(x)=q(-x)$, and $q$ is symmetric under reflection in the $x$-axis and rotation by $\pi/2$ about the origin. Moreover, $q$ is $\mathcal{C}^3$ and there exist $c > 0$ and $\varepsilon>0$ such that, for every multi-index $\alpha$ with $|\alpha| \leq 3, |\partial^\alpha q(x)| \leq c|x|^{-(1+\varepsilon)}$. Furthermore, the support of the Fourier transform of $q$ contains an open set.\footnote{Note that this is a direct consequence of $q$ being not identically equal to $0$ when $q$ is $L^1$, which will be case most of the time for us since we will often assume that Condition \ref{cond:pol} holds for some $\beta > 2$.}
\end{condition}

\begin{condition}[Weaker than Condition \ref{cond:sym}]\label{cond:weak}
The function $q$ is $\mathcal{C}^2$ and there exist $c,\delta > 0$ such that, for every multi-index $\alpha$ with $|\alpha| \leq 2$, we have
\[
\forall x \in \R^2, \, \left| \partial^\alpha q (x) \right| \leq c|x|^{-(1+\delta)} \, .
\]
\end{condition}

\begin{condition}[Positive correlations]\label{cond:FKG}
The function $q$ satisfies
\[
q \star q \geq 0 \, .
\]
\end{condition}

\begin{condition}[Polynomial decay; depends on a parameter $\beta > 0$]\label{cond:pol}
There exists $c > 0$ such that, for every multi-index $\alpha$ with $|\alpha| \leq 1, |\partial^\alpha q(x)| \leq c|x|^{-\beta}$.
\end{condition}

Assume that $q$ satisfies Conditons \ref{cond:sym}, \ref{cond:FKG} and \ref{cond:pol} for some $\beta > 2$. Let $f$ be the planar Gaussian field
\[
x \in \R^2 \mapsto f(x)=q\star W(x) := \int_{\R^2} q(x-y)dW(y)\, .
\]
Notice that the covariance of $f$ is $(x,y) \in (\R^2)^2 \mapsto q \star q (x-y)$. Moreover, one can show (see for instance Subsection 3.2 of \cite{MV}) that Condition \ref{cond:sym} implies that there exists a modification of $f$ which is a.s. $\mathcal{C}^2$. In the rest of the paper, we work with this $\mathcal{C}^2$ modification (N.B. the weaker Condition \ref{cond:weak} only implies the property that there exists a modification of $f$ which is continuous, see for instance Appendix \ref{s.appendix}). Note that the Bargmann-Fock field can be realized by choosing $q(x)=(2/\pi)^{1/4}e^{-|x|^2}$ (which satisfies all the above conditions, and for any $\beta >0$).

We extend our above noise sensitivity result to the dynamical processes $t \mapsto f(t)$ defined as follows. Let $(W_t(dx))_{t \geq 0}$ be a planar white noise driven by an Ornstein-Uhlenbeck dynamics. More precisely, we consider a centered Gaussian process $\left( \int u dW_t \right)_{u,t}$ indexed by $(u,t) \in L^2(\R^2) \times \R_+$ with covariance
\[
\E \left[ \int u dW_t \int v dW_s \right] = e^{-|t-s|} \int uv \, .
\]
For every $t \geq 0$, let
\[
f(t) = q \star W_t \, .
\]
As shown in the appendix, if we assume that $q$ satisfies Condition \ref{cond:weak}, then there exists a modification of $(t,x) \in \R_+ \times \R^2 \mapsto f(t,x)$ that is continuous. In the following, we consider such a modification. Note that, in the case of the Bargmann-Fock field, this dynamics is the same as the dynamics from Subsection~\ref{ss:BF}.
\begin{remark}\label{rem:coupl}
For any $t \geq 0$, one may realize the joint coupling $(f(0), f(t))$ as follows:
\[
f(t) = e^{-t} f(0) + \sqrt{1-e^{-2t}} \,  \widetilde{f}\,,
\]
where $\widetilde{f}$ is an independent copy of $f(0)$. 
\end{remark}
The following theorem generalizes Theorem \ref{th.main}.
\begin{theorem}\label{thm:mainbis}
The content of Theorem \ref{th.main}  extends to any $q$ satisfying Conditions \ref{cond:sym}, \ref{cond:FKG} and \ref{cond:pol} for some $\beta > 2$ (with constants $c$ and $\alpha$ that may depend on $q$).
\end{theorem}

In Section \ref{s:PT}, we prove a sharp threshold result by relying on the above noise sensitivity result as well as an idea originating from \cite{BKS} which requires the analysis carried in Section \ref{s.cor}. Before stating our result, here is a short overview of the current ``state of the art'' on the phase-transition of planar Gaussian fields. Assume that $q$ satisfies Conditions \ref{cond:sym}, \ref{cond:FKG} and \ref{cond:pol} for some $\beta > 2$. Then,
\begin{itemize}
\item Theorem 1.11 of \cite{MV} (see \cite{BG,BM,RVa} for the same result with stronger assumptions on $\beta$). For every quad $Q$, there exists $c=c(Q,q) \in ]0,1[$ such that, for every $s > 0$, $c \leq \Pro \left[ \cross_0(sQ) \right] \leq 1-c$;
\item Theorem 1.7 of \cite{R} (see \cite{MV} for the same result with the less general assumption $q \geq 0$ instead of Condition \ref{cond:FKG}: $q \star q \geq 0$).\footnote{Actually, the results from \cite{R} are more general in the sense that they require less regularity conditions on $q$. However, in the present paper, our focus is more on positivity conditions and conditions about the speed of decay of $q$.} For every $\ell > 0$ and every quad $Q$, there exists $c=c(\ell,Q,q)>0$ such that, for every $s>0$, $\Pro \left[ \cross_{-\ell}(sQ) \right] \leq e^{-cs}$ and $\Pro \left[ \cross_\ell(sQ) \right] \geq 1 -e^{-cs}$;
\item Theorem 1.15 of \cite{MV}. If we assume furthermore that $q \geq 0$, then the near-critical window is polynomially small in the sense that there exists $\alpha=\alpha(q) > 0$ such that, for every quad $Q$, $1-\Pro \left[ \cross_{s^{-\alpha}}(sQ) \right]$ and $\Pro \left[ \cross_{-s^{-\alpha}}(sQ) \right]$ go to $0$ as $s \rightarrow +\infty$.
\end{itemize}

As in the case of the Bargmann-Fock field, the two first items imply that i) if $\ell \leq 0$ then a.s. $\{ f > -\ell \}$ has no unbounded component and ii) if $\ell > 0$, then $\{ f > -\ell \}$ has a unique unbounded component.

In the present paper, we generalize the above results by combining in some sense the conclusions of \cite[Theorem 1.15]{MV} with those of \cite[Theorem 1.7]{R}. I.e. i) we obtain a new proof of the sharp threshold result from \cite{R} (see Section \ref{s:PT}), and ii) we obtain that the near-critical window is polynomially small (see Theorem \ref{thm:pol} below) when $q \star q \geq 0$ (rather than $q \geq 0$).
\begin{theorem}\label{thm:pol}
If $q$ satisfies Conditons \ref{cond:sym}, \ref{cond:FKG} and \ref{cond:pol} for some $\beta > 2$, then there exists $\alpha=\alpha(q) > 0$ such that, for every quad $Q$, $1-\Pro \left[ \cross_{s^{-\alpha}}(sQ) \right]$ and $\Pro \left[ \cross_{-s^{-\alpha}}(sQ) \right]$ go to $0$ as $s \rightarrow +\infty$.
\end{theorem}

\subsection{$L^2$ versus $L^1$ methods}
 
In this subsection, we wish to compare briefly the methods to prove sharpness results for such models from the present paper and \cite{RVb,MV,R}. In \cite{RVb,MV}, the main intermediate result is the proof that the following quantity goes to $+\infty$ uniformly in $\ell \in \R$
\begin{equation}\label{e:deriv_var}
\frac{d \Pro \left[ \cross_\ell(nQ) \right]}{d\ell} \times \frac{1}{\Pro \left[ \cross_\ell(nQ) \right] (1-\Pro \left[ \cross_\ell(nQ) \right])} \, .
\end{equation}
In \cite{RVb} and \cite{MV}, discretization procedures are used in order to apply respectively a KKL type theorem and the OSSS inequality.\footnote{The KKL theorem and the OSSS inequality have been used to detect phase transitions for numerous statistical physics models, see in particular \cite{BR} and \cite{OSSS}.} If $\cross_\ell^\varepsilon(nQ)$ is the discrete version of $\cross_\ell(nQ)$, it was roughly obtained - respectively in \cite{RVb} and \cite{MV} - that \eqref{e:deriv_var} was less than $O(1/(\sqrt{\log(n)} \varepsilon))$ and $O(n^{-c}/\varepsilon)$. In particular, these estimates are useless in the limit $\varepsilon \rightarrow 0$ so it was necessary to be very quantitative on the discretization procedure.

In \cite{R}, Rivera uses a Talagrand inequality - which is an analogue of the KKL theorem - and obtains an estimate (not on \eqref{e:deriv_var} but on another suitable quantity) \textbf{uniform in $\varepsilon$}. The fact that this estimate behaves well when one passes to the limit may come from the fact that Talagrand's inequality is an inequality on the $L^2$ norm of the gradient (while KKL type inequalities are estimates on the ``sum of influences'' which is the $L^1$ norm of the gradient for Gaussians, see \cite{KMS}), which may correspond more to the Gaussian setting. In the present paper, we also obtain an estimate uniform in $1/N$ (which has to be interpreted as the discretization mesh $\varepsilon$), see Step~4 in the proof of Theorem \ref{th.main} written in Section \ref{s.proof}. Similarly, we can interprete this by noting that the Schramm-Steif theorem (see Subsection \ref{ss:Sch-St}) is an $L^2$ estimate whereas the OSSS inequality - which is also an estimate involving randomized algorithms - is an estimate on the $L^1$ norm of the gradient.

Let us finally note, in the case of planar percolation models, that KKL and Talagrand type inequalities give a logarithmic upper bound on the size of the near-critical window while the OSSS and Schramm-Steif inequalities give sharper polynomial upper-bounds.

\subsection{The main tool: the Schramm-Steif randomized algorithm theorem}\label{ss:Sch-St}

The main tool of our proof is the Schramm-Steif randomized algorithm theorem \cite{SS}. Let us recall this result. We refer to \cite{book} for more details on Boolean functions and noise sensitivity. Let $n \in \N^*$ and consider the hypercube $\Omega_n =\{-1,1\}^n$. We equip $\Omega_n$ with the uniform probability measure $P$ and we consider the Fourier-Walsh basis $(\chi_S)_{S \subseteq \{1,\cdots,n\}}$ which is the orthonormal basis of $L^2(\Omega_n,P)$ defined by
\[
\chi_S(\sigma) = \prod_{i \in S} \sigma_i \, .
\]
Every function $g \, : \, \Omega_n \rightarrow \R$ can be decomposed in a unique way as
\[
g = \sum_{S \subseteq \{1,\cdots,n\}} \widehat{g}(S) \chi_S \, .
\]
If the dynamics $t \mapsto \sigma(t)$ is defined by sampling a configuration $\sigma(0) \sim P$ and by resampling each bit independently at rate $1$, then we have
\[
\Cov{ g(\sigma(0)),g(\sigma(t)) } = \sum_{S \neq \emptyset} \widehat{g}(S)^2 e^{-t|S|} \, .
\]
Let us recall that a sequence of Boolean function $g_n \, : \, \Omega_{m_n} \rightarrow \{0,1\}$ for some sequence $(m_n)_{n \in \N}$ that goes to $+\infty$ is said noise sensitive (\cite{BKS}) if
\[
\forall t>0, \, \Cov{ g_n(\sigma(0)),g_n(\sigma(t)) } \underset{n \rightarrow +\infty}{\longrightarrow} 0 \, .
\]
Noise sensitivity has been proved for discrete percolation \cite{BKS,SS,GPS} and for some continuous percolation models such as the Poisson Boolean model \cite{ABGM} and Voronoi percolation \cite{AGMT,AB}. In these two last works, the approach relies on a result by Schramm and Steif. In order to state Schramm-Steif theorem, we need to recall what is a (randomized) algorithm. If $g \, : \, \Omega_n \rightarrow \R$, a randomized algorithm that determines $g$ is a procedure that asks the values of the bits $i \in \{1,\cdots,n\}$ step by step where at each step the algorithm can ask for the value of one or several bits and the choice of the new bit(s) to ask is based on the values of the bits previously queried. The first bit may be random (and extra randomness can be used to decide what is the next bit queried but we do not need this in the present paper). We also ask that the algorithm stops once $g$ is determined. The revealment of the algorithm is the supremum on every bit $i$ of the probability that $i$ is required by the algorithm. The revealment $\delta(g)$ of $g$ is the infimum of the revealements of all the algorithms that determine $g$.
\begin{theorem}[\cite{SS}]\label{thm:SchSt}
For every $g \, : \, \Omega_n \rightarrow \R$ and every $k \in \N^*$, we have
\[
\sum_{S \, : \, |S|=k} \widehat{g}(S)^2 \leq k\delta(g) \int g^2 dP \, .
\]
\end{theorem}
We will use this theorem as follows: in Section \ref{s.proof}, we will approximate the white noise by a discrete white noise with $\pm 1$ bits and we will observe that running the above dynamics on the bits of the discrete white noise is the same (in the limit) as the Ornstein-Uhlenbeck dynamics from Section \ref{ss:gen}. Applying Schramm-Steif theorem to the $\pm 1$ bits and estimating the revealment thanks to estimates from \cite{MV} will give the noise sensitivity result. Actually, in order to define a suitable algorithm, we will have to work with a truncated (i.e. finite range) version of our field.

\subsection{A motivation: exceptional times and exceptional planes}\label{ss.motiv}

Our initial motivation in studying the noise sensitivity of Bargmann-Fock percolation was not our above application to sharp thresholds (Theorem \ref{thm:pol}) but rather to establish the existence of \textbf{exceptional times} for different natural dynamics on Bargmann-Fock percolation on $\R^2$ which are listed below. One of the long-term goals is to prove that, if one considers a Bargmann-Fock field in dimension $3$, then a.s. there exist ``exceptional'' planes $P \subset \R^3$ in which there is an unbounded nodal line (see Conjecture \ref{c.2} for a more precise statement).

In each case, as we explain below we still miss at least one key ingredient in order to prove the existence of exceptional times. 

\medskip

\ni
\textbf{$i)$ Ornstein-Uhlenbeck dynamics on Bargmann-Fock}. We already considered this dynamics above. Using the above notations, recall it can be defined as 
$f(t,x_1,x_2)=e^{-\frac{1}{2}|x|^2} \sum_{i,j\in \N} \frac{a_{i,j}(t)}{\sqrt{i! j!}} x_1^i x_2^j$ or as $f(t)=q\star W_t$ where $q(x)=(2/\pi)^{1/4}e^{-|x|^2}$. 

From \cite{Alex,BG}, it is known that for any fixed $t$, a.s. there is no unbounded connected component neither in $\calC_t:=\{x\in \R^2, f(t,x) > 0 \}$ nor in $\calC_t^*:=\{x\in \R^2, f(t,x) = 0 \}$. Our motivation was to study the question of existence of \textbf{exceptional times} that can be defined for instance as the (random) times $t$ at which there is an unbounded component in $\calC_t$ or those for which this is the case for $\calC_t^*$ (i.e. the times at which there is an unbounded nodal line). By analogy with site-percolation on the triangular grid, we conjecture that the following happens. 

\begin{conjecture}\label{c.1}
A.s. exceptional times at which there exists an unbounded component in $\calC_t$ do exist. Moreover, a.s. the Hausdorff dimension of this random set of times  is $\frac{67} {72}$ (N.B. for dynamical percolation on the triangular lattice, it was shown in \cite{GPS} that the analogous dimension is a.s. $\frac {31} {36}$).

Furthermore, exceptional times at which there exists an unbounded nodal line also exist a.s. and the corresponding Hausdorff dimension is $\frac{5}{6}$ (N.B. for dynamical percolation on the triangular lattice, it was shown in \cite{SS,GPS} that the analogous dimension is an a.s. constant that lies in $[\frac{1}{9},\frac {2} {3}]$ and it is conjectured that this constant is $\frac{2}{3}$).
\end{conjecture}

We are far from being able to prove this conjecture. Here is why: already for the classical dynamical percolation model on the square lattice $\Z^2$, it is not known up to now how to prove the existence of exceptional times using the randomized algorithms techniques from \cite{SS} (the only proof for $\Z^2$ is provided in \cite{GPS} and would not extend easily to Bargmann-Fock). In fact the situation is worse for Bargmann-Fock than on $\Z^2$: indeed, in order to define a suitable algorithm on the white noise bits, we have to work with a finite range version of the Bargmann-Fock field. Let $m_n=C\sqrt{\log(n)}$ with $C$ very large. By truncating $q$ and by using estimates from \cite{MV} (see Subsection 3.4 therein), the Bargmann-Fock field can be approximated on a quad $nQ$ by a $m_n$-dependent field $f_{\text{trunc}}$, but we cannot obtain less spatial dependencies. As a result, the bound we can get on the revealment is not as good as for $\Z^2$ because, if one wants to reveal $f_{\text{trunc}}(x)$ then one has to reveal \textbf{all the bits of the white noise at distance $m_n$ from $x$}. Moreover, one does not have {\em separation of arms} tools for Bargmann-Fock percolation. Because of these reasons, it looks out of reach at the moment to prove exceptional times for this dynamical model.
\medskip

Let us now briefly explain why we expect a $67/72$-Hausdorff dimension instead of the classical one $31/36$.  There are two ways to see where the difference comes from. 1) If one is looking for an upper-bound on the Hausdorff dimension, then one may proceed exactly as in \cite{SS} by dividing the unit-interval of times $[0,1]$ into $\eps^{-1}$ intervals of length $\eps$. Then, on each of these intervals, one tries to have an upper bound on the probability that this interval contains an exceptional time. Usually one proceeds by relying on an easy stochastic domination. In this case though, one cannot hope to stochastically dominate $\bigcup_{0\leq u \leq \eps} \{x \in \R^2: f(u,x) > 0\}$ by $\{x \in \R^2: f(0,x) > -\lambda(\eps)\}$ for some small and well-chosen $\eps \mapsto \lambda(\eps)$ (this is due to the fact that at  large distances there will be arbitrary large fluctuations). Yet, Appendix \ref{s.appendix} suggests that for any $a>0$, $\lambda(\eps):=\eps^{1/2-a}$ would give an ``almost'' such stochastic domination. If one believes that later non-trivial fact plus the believed same universal behavior for the near-critical Gaussian percolation process $\ell \mapsto \{x : f(0,x)>-\ell\}$ as for site-percolation on the triangular lattice,\footnote{For site-percolation on the triangular lattice, it is known that at parameter $p_c+\eps$, the probability that the origin is in an unbounded cluster behaves like $\eps^{5/36}$, see \cite{SW}.} then we obtain that an $\eps$-interval of time should contain an exceptional time with probability at most $\eps^{1/2 \times 5/36}$. This implies the expected bound $1-5/(2 \times 36) = 67/72$ on the Hausdorff dimension (for more details on similar computations, see for instance the beginning of Section 6 of \cite{SS}). 2) The second reason is that by a close inspection of $\frac d {dt} \Pb{f(0), f(t) \in  \cross_0(RQ)}$,\footnote{The study of such a derivative in terms of pivotal events can be extracted from works of Piterbarg \cite{Pit}, see also the more recent approach in Section 2 of \cite{RVa} or the continuum analogue in \cite{BMR}.} it appears that decorrelation should start to happen when $t^{-1/2} \approx R^2 \alpha_4^{\mathrm{BF}}(R)$ where $\alpha_4^{\mathrm{BF}}(R)$ denotes the probability of the $4$-arm event from scale $1$ to scale $R$ for the Bargmann-Fock field (which is the event that there are $4$ paths of alternating sign from the ball $B_1(0)$ to $\partial B_R(0)$). By universality, $R^2 \alpha_4^{\mathrm{BF}}(R)$ is believed to be of order $R^{3/4}$. This computation suggests that the dynamical correlation length for this O.U. dynamics will be $t^{-2/3}$ (instead of $t^{-4/3}$ in Bernoulli percolation), which suggests that the upper-bound given just above might be the true Hausdorff dimension (indeed, this upper-bound equals $1-2/3 \times 5/48$ where $5/48$ is the exponent of the $1$-arm event, see \cite{LSW}).
\medskip

The study of the set of exceptional times at which there is an unbounded nodal line is harder since this event is not monotonic. However, similar observations as above but applied to the $2$-arm event (see Section 8 of \cite{SS} for instance) suggest that the Hausdorff dimension of this set of exceptional times is $1-2/3 \times 1/4=5/6$ (indeed, it has been proved in \cite{SW} that the exponent of the $2$-arm event for site-percolation on the triangular lattice is $1/4$).

\smallskip
\ni
\textbf{$ii)$ Exceptional planes for Bargmann-Fock field in $\R^3$.} 
 Consider now a 3D Bargmann-Fock field $F$ on $\R^3$, and for each $t\in \R$, let $f^\mathrm{hor}(t)$ be the two-dimensional Bargmann-Fock field obtained by restricting $F$ to the horizontal plane $\{(x,y,t)\}_{x,y\in \R}$. We are interested in the following (non-Markovian) dynamics:
\[
t \mapsto f^\mathrm{hor}(t)\,.
\]
It is easy to check that for every $t \in \R$ the joint coupling $(f^{\mathrm{hor}}(0),f^{\mathrm{hor} }(t))$ can be realized as follows:
\begin{align*}\label{}
f^\mathrm{hor}(t) & = e^{-\frac {t^2} 2} f^\mathrm{hor}(0) + \sqrt{1 - e^{-t^2}}\, \widetilde{f}^\mathrm{hor} \, ,
\end{align*}
where $\widetilde{f}^\mathrm{hor}$ is an independent copy of $f^\mathrm{hor}(0)$. In particular, we see here that this dynamics is locally much slower than the above O.U. dynamics on Bargmann-Fock.  
Despite this slowing down, we claim that a proof of the above conjecture would imply the following one:
\begin{conjecture}\label{c.2}
A.s. the Hausdorff dimension of the set of times at which there is an unbounded component in $\{ f^\mathrm{hor}(t) > 0 \}$ is $31/36$ (N.B. same as for the triangular lattice). Moreover, a.s. the Hausdorff dimension of the set of times at which there is an unbounded component in $\{ f^\mathrm{hor}(t) = 0 \}$ is $2/3$.
\end{conjecture}

(Note this would imply in particular the existence of an unbounded nodal surface for $F$, Bargmann-Fock field on $\R^3$.)

Let us explain briefly why we expect this smaller dimension compared to the O.U. case and why this should follow from a proof of Conjecture \ref{c.1}. As explained in \cite{SS} for standard dynamical percolation, the value of the Hausdorff dimension follows directly from the knowledge of the ``two-point function'', $\Pb{0\overset{\omega_0}\longleftrightarrow \p B_R \text{ and }0\overset{\omega_t}\longleftrightarrow \p B_R}$, where $(\omega_t)_{t \geq 0}$ is a dynamical percolation process. In the case of the triangular grid, this two-point function is shown in \cite{GPS} to behave as $t^{-5/36} \alpha_1(R)^2$, where $\alpha_1(R)$ is the probability of the $1$-arm event. From the above discussion in $i)$, we  can expect that the two-point function for O.U. dynamics rather behaves as $(t^{1/2})^{-5/36} \alpha_1(R)^2$. If so, not only it would imply Conjecture \ref{c.1} but also, thanks to the above identity for the joint coupling $(f^{\mathrm{hor}}(0),f^{\mathrm{hor} }(t))$, it would imply 
\begin{align*}\label{}
\Pb{0 \overset{f^{\mathrm{hor}}(0)}\longleftrightarrow \p B_R \text{ and } 0 \overset{f^{\mathrm{hor}}(t)}\longleftrightarrow \p B_R} & = (t^{2/2})^{-5/36} \alpha_1^{\mathrm{BF}}(R)^2 \\
& = t^{-5/36}\alpha_1^{\mathrm{BF}}(R)^2 \,,
\end{align*}
where $\alpha_1^{\mathrm{BF}}(R)$ is the probability of the $1$-arm event for the Bargmann-Fock field. As in \cite{SS,GPS}, and if we believe same universal behavior of the probability of the $1$-arm event as for site-percolation on the triangular lattice, this estimate would readily imply the first part of Conjecture \ref{c.2}. Analogous (but harder because of the non-monotonicity of the $2$-arm event) arguments would imply the second part of this conjecture.
\medskip

Finally, in order to detect interesting exceptional times, let us point out that one may also try to integrate further on the angle of planes in the spirit of \cite{BS}.

\medskip
\paragraph{\textbf{Notations.}} We use the following notations for $\sigma$-algebras: $\mathcal{F}$ is the usual $\sigma$-algebra on the set of continuous functions from $\R^2$ to $\R$ ($\mathcal{F}$ is generated by the functions $u \mapsto u(x)$ for every $x \in \R^2$). Moreover, for every subset $D$ of $\R^2$, we let $\mathcal{F}_D$ denote the $\sigma$-algebra generated by $u \mapsto u(x)$ for every $x \in D$.

Finally, we denote by $O(1)$ a positive bounded function, by $\Omega(1)$ a positive function bounded away from $0$ and by $\Theta(1)$ a positive function bounded away from $0$ and $+\infty$.

\medskip
\ni
\textbf{Acknowledgments.}

\ni
We wish to thank Stephen Muirhead and Alejandro Rivera for very fruitful discussions about randomized algorithms in the context of smooth Gaussian percolation. We also would like to thank Vincent Beffara, Charles-Edouard Bréhier, Damien Gayet and Avelio Sep\'{u}lveda for very interesting discussions. Moreover, we would like to thank Daniel Contreras Salinas, Stephen Muirhead and Alejandro Rivera for useful remarks on a preliminary version of this paper. Finally, we wish to thank an anonymous referee for helpful comments.\\
This research has been partially supported by the 
ANR grant \textsc{Liouville} ANR-15-CE40-0013 and the ERC grant LiKo 676999.

\section{Proof of noise sensitivity}\label{s.proof}

In this section, we prove Theorem \ref{thm:mainbis} (and as a byproduct Theorem \ref{th.main} which is a particular case). We will rely on Schramm and Steif randomized algorithm theorem and on estimates from \cite{MV} (more precisely, we will use Sections 3 and 4 of \cite{MV} but not Section 5, where another randomized algorithm approach is used, based on the OSSS inequality rather than on \cite{SS}). Let $q$ satisfying Conditions \ref{cond:sym}, \ref{cond:FKG} and \ref{cond:pol} for some fixed $\beta>2$, let $f$ be the $\mathcal{C}^2$ random Gaussian field $q \star W$, and let $Q$ be a quad. The proof is divided into the following steps.
\medskip

\paragraph{\bf Step 1.} We first observe that, by linearity and Remark \ref{rem:coupl}, we have the following useful rewriting of $(f(0), f(t))$,
\[
\begin{cases}
f(0) & = q \star W \\
f(t) & = e^{-t }q\star W + \sqrt{1 - e^{-2t}} q\star \widetilde{W}
\end{cases}
\]
where $\widetilde{W}$ is an independent copy of the white noise $W$.
\medskip

\paragraph{\bf Step 2.} In \cite{MV}, the following local approximation of the field $f$ is introduced in order to have spatial independency: for any radius $r\geq 1$,  
\[
f_r := q_r \star W\,,
\]
where $q_r(x) = \chi_r(x) q(x)$ and $\chi_r \, : \, \R^2 \rightarrow [0,1]$ is a smooth approximation of $1_{ |\bullet| > r}$. More precisely, we ask that $\chi_r$ is smooth, isotropic, that for every $k \geq 1$, the $k^{th}$ derivatives of $\chi_r$ are uniformly bounded, and that
\begin{eqnarray*}
\chi_r(x) = 1 \text{ if } |x| \leq r/2-1/4 \, ;
\chi_r(x) = 0 \text{ if } |x| \geq r/2 \, .
\end{eqnarray*}
(Note that either $q_r$ is identically equal to $0$ or $q_r$ satisfies Conditions \ref{cond:sym}, \ref{cond:FKG} and \ref{cond:pol} since $q$ does.) In our setting, we are interested in connection events at large scale $n$ (such as $\{f \in \cross_0(nQ)\}$). It will be convenient at these scales to rely on the approximation
\[
f \approx f_{n^{\gamma}} \text{ if } \gamma > \frac 1 {\beta-1} \, .
\]
This approximation is robust for any monotonic event as can be seen from the following proposition.

\begin{proposition}[Proposition 3.11 and Corollary 3.7 of \cite{MV}]\label{p:trunc}
For every $\gamma>0$, there exists $c=c(q,\beta,\gamma)$ such that, for every $R \geq 2$, every $D$ subset of $\R^2$ with diameter less than $R$, and every monotonic event $A \in \mathcal{F}_D$, we have
\[
\Pb{ \{ f \in A \} \Delta \{ f_{R^{\gamma}} \in A \} } \leq c \log(R) R^{1+\gamma(1-\beta)} \, .
\]
\end{proposition}
Let us fix some exponent $\gamma \in ]0,1[$ such that $1+\gamma(1-\beta)<0$, which is possible since $\beta>2$. Thanks to Proposition \ref{p:trunc}, it is enough for us to prove the following estimate for some $\alpha>0$ and $c>0$
\begin{align*}\label{}
\Cov{1_{f_{n^\gamma}(0) \in \cross_0(nQ)}, 1_{f_{n^\gamma}(t_n)\in \cross_0(nQ)}} \leq  cn^{-\alpha}
\end{align*}
as soon as $t_n \geq n^{-\alpha}$.
\medskip

\paragraph{\bf Step 3.} We now proceed to a further approximation step, where one approximates the Gaussian white noise $W(dx)$ using independent Bernoulli variables. This second approximation step is reminiscent to the definition of $f_r^\varepsilon$ in \cite{MV} except that we rely on Bernoulli variables here instead of Gaussian variables. As such we are as close as we may from the setup in \cite{SS}. 

Let $N \geq 1$ be an integer and let us consider the following discrete white noise on $\R^2$
\begin{align*}\label{}
W^N(x) := N \sum_{v\in \frac 1 N \Z^2} \sigma_v 1_{x \in v+[-1/N,1/N]^2} \,,
\end{align*}
where the random variables $\sigma_v$ are independent and $\Pro \left[ \sigma_v = 1\right] = \Pro \left[ \sigma_v = -1 \right]=1/2$. We thus define 
\[
f_{n^\gamma}^N:= q_{n^\gamma} \star W^N \, .
\]
Notice that the indicator function $1_{f_{n^\gamma}^N \in \cross(nQ)}$ is nothing but a \textbf{Boolean function} defined on a hypercube $\Omega_{n,N}=\{-1,1\}^{\Theta(n^2 N^2)}$ (N.B. $\Theta(n^2N^2)$ comes from the fact that there are of order $n^2\,N^2$ Bernoulli bits in the $n^\gamma$-neighbourhood of the rescaled quad $nQ$).

For any open square $S \subset \R^2$ and any $s\in \R$, let $H^s(S)$ denote the Sobolev space on $S$ of order $s$. Another important remark at this stage is that if we let each Bernoulli variable $\sigma_v$ evolve according to a rate 1 Poisson Point process (i.e. $t \mapsto \sigma_v(t)$ switches its state independently at rate 1 for each $v\in \frac 1 N \Z^2$) then it induces a dynamics $t \mapsto W^N_t$ which is such that the following holds: For every $t \geq 0$, every $\varepsilon>0$ and every (open) square $S \subset \R^2$, we have the following convergence in law in the space $H^{-1-\varepsilon}(S) \times H^{-1-\varepsilon}(S)$:
\begin{align}\label{e:law_sobo}
(W^N_0, W^N_t) \overset{\text{law}}\longrightarrow (W_0,W_t)\,,
\end{align}
where $W_t = e^{-t}W+ \sqrt{1- e^{-2t}} \widetilde{W}$, $\widetilde{W}$ is a white noise independent of $W$. See Appendix \ref{s:sobo} for a proof of this (classical) fact.

For every $t \geq 0$, let
\[
f_{n^\gamma}(t)=q_{n^\gamma} \star W_t \text{ and } f_{n^\gamma}^N(t)=q_{n^\gamma} \star W_t^N.
\]
Let us end this step by showing the following consequence of \eqref{e:law_sobo}.
\begin{lemma}\label{l:coupl_cv}
For every $t \in \R_+$ and every $n$, we can couple $(W^N_0,W^N_t)_{N\in \N}$ and $(W_0,W_t)$ such that a.s. the following holds
\[
\|f_{n^\gamma}(0) - f_{n^\gamma}^N(0) \|_{\infty, nQ} \underset{N\to \infty}\longrightarrow 0 \text{ and } \|f_{n^\gamma}(t) - f_{n^\gamma}^N(t) \|_{\infty, nQ} \underset{N\to \infty}\longrightarrow 0 \, .
\]
\end{lemma}
Note that Lemma \ref{l:coupl_cv} implies that
\begin{multline}\label{e.cv_of_cov}
\Cov{1_{f_{n^\gamma}^N(0) \in \cross_0(nQ)}, 1_{f_{n^\gamma}^N(t)\in \cross_0(nQ)}}\\
 \underset{N\to \infty}\longrightarrow \Cov{1_{f_{n^\gamma}(0) \in \cross_0(nQ)}, 1_{f_{n^\gamma}(t)\in \cross_0(nQ)}} \, .
\end{multline}
Actually, to deduce this result from Lemma \ref{l:coupl_cv}, one needs a regularity result about $f_{n^\gamma}$. We refer to Appendix \ref{s.reg} for more details.
\begin{proof}[Proof of Lemma \ref{l:coupl_cv}]
One difference here with \cite{MV} is that the discrete white noise field $W^N$ is less naturally coupled to $W$. Fix some $n$ and consider some (open) square $S$ that contains the $n^\gamma$-neighbourhood of $nQ$. Let $t \in \R_+$. By \eqref{e:law_sobo} and by relying on Skorokhod's representation theorem, one may couple $(W^N_0,W^N_t)_{N\in \N}$ and $(W_0,W_t)$ on the same probability space so that
\[
\| W^N_0 - W_0 \|_{H^{-1-\varepsilon}(S)} \to 0 \, ; \: \| W^N_t - W_t \|_{H^{-1-\varepsilon}(S)} \to 0 \, .
\]
Now, using the fact that for any $\calC^2$ function $h$ with compact support included in $S$ we have 
\begin{align}\label{eq.duality_formula}
|\<{W^N-W,h}| \leq \|W^N - W\|_{H^{-2}(S)} \|h\|_{H^2(S)} \, ,
\end{align}
we readily conclude (by considering, for any $x \in nQ$, the function $h \, : \, y \in S \mapsto h(y)=h^x(y):=q_{n^\gamma}(x-y)$).
\end{proof}
\medskip

\paragraph{\bf Step 4.} Let $g_{n,N} \, : \, \Omega_{n,N} \rightarrow \{0,1\}$ be the indicator function of the crossing event of the quad $n Q$ for $f_{n^\gamma}^N$. Schramm and Steif (Theorem \ref{thm:SchSt}) show that:
\begin{eqnarray*}
\Cov{1_{f_{n^\gamma}^N(0) \in \cross_0(nQ)}, 1_{f_{n^\gamma}^N(t)\in \cross_0(nQ)}} & = & \Cov{g_{n,N}(\sigma(0)), g_{n,N}(\sigma(t))}\\
& = & \sum_{S \neq \emptyset} \hat g_{n,N}(S)^2 e^{-t|S|}\\
& \leq & \sum_{k \geq 1} k e^{-tk} \delta( g_{n,N} ) \, ,
\end{eqnarray*}
where $\delta( g_{n,N} )$ is the revealment of $g_{n,N}$. Note that, by \eqref{e.cv_of_cov}, it is now sufficient to show that there exists some $\delta_n$ that decays at least polynomially fast and such that, for every $n$, $\delta(g_{n,N})$ converges to $\delta_n$ as $N\rightarrow +\infty$. Let us prove this. Let us first define a (randomized) algorithm that determines $g_{n,N}$. In this definition, ``discovering'' a region of the plane means that we reveal all the bits $\sigma_v$ at distance less than $n^{\gamma}/2$ from this region (remember that we have fixed some $\gamma \in ]0,1[$ such that $1+\gamma(1-\beta)<0$). Since $q_{n^\gamma}(x)=0$ for every $x \in \R^2$ such that $|x| \geq n^\gamma/2$, this gives us the value of $f^N_{n^\gamma}(x)$ for every $x$ in this region. Let us first define the algorithm in the case where $Q=[0,1]^2$ and where the left and right sides are the two distinguished segments:

Choose uniformly at random some $k \in \{1,\cdots,\lfloor n^{1-\gamma} \rfloor \}$ and discover the segment $\{ kn^\gamma \} \times \R \cap nQ$. Then, discover all the $1 \times 1$ squares of the grid $\Z^2$ (for instance) that contain a point that is connected to $\{ kn^\gamma \} \times \R$ by a path included in the intersection of the region already explored and $\{ f^N_{n^\gamma} > 0 \} \cap nQ$. Stop the algorithm when all the connected components of $\{ f^N_{n^\gamma} > 0 \} \cap nQ$ that intersect $\{ kn^\gamma \} \times \R$ have been discovered. Note that this algorithm determines the crossing event.

Let $\arm_0(r,R)$ denote the event that there is a positive continuous path included in $[-R,R]^2 \setminus ]-r,r[^2$ that crosses this annulus (if $r>R$, the convention is that $\arm_0(r,R)$ is the sure event). There exists $c>0$ such that the revealment of the above algorithm is less than or equal to
\[
c n^{\gamma-1} \sum_{k=1}^{\lfloor n^{1-\gamma} \rfloor} \Pro \left[ f_{n^\gamma}^N \in \arm_0(cn^\gamma,(k-1)n^\gamma/c) \right] \, .
\]
By Lemma \ref{l:coupl_cv} (and, once again, Appendix \ref{s.reg}), the above converges to
\[
c n^{\gamma-1} \sum_{k=1}^{\lfloor n^{1-\gamma} \rfloor} \Pro \left[ f_{n^\gamma} \in \arm_0(cn^\gamma,(k-1)n^\gamma/c) \right] \, .
\]
Remember that, in order to prove Theorem \ref{thm:mainbis}, it is sufficient to show that the above decays at least polynomially fast in $n$. By Proposition \ref{p:trunc}, it is enough to show that $\Pro \left[ f \in \arm_0(r,R) \right]$ decays at least polynomially fast in $(r/R)$. This is given by Theorem 4.7 of \cite{MV}.

Let us end this section by explaining how to generalize this to any quad $Q$. Let $\eta,\eta'$ be the two distinguished segments of $Q$. First note that there exist $h=h(Q)>0$ and $m \geq 2$ such that there exist $m$ unit squares of the grid $h\Z^2$: $S_1,\cdots,S_m$ that satisfy i) $S_1$ intersects $\eta$ but not $\partial Q \setminus \eta$, ii) $S_m$ intersects $\eta'$ but not $\partial Q \setminus \eta'$, iii) all the $S_i$'s are distinct and $S_{i+1}$ shares a side with $S_i$ for every $i \in \{1,\cdots,m-1 \}$, iv) $\cup_{i=2}^{m-1} S_i \subseteq Q$. We then run exactly the same algorithm as in the case $Q=[0,1]^2$ but by replacing the line $\{ k \} \times \R$ by the $\Theta(n^{1-\gamma})$ lines included in $\cup_{i=1}^m (nS_i)$ depicted in Figure \ref{fig:algo}. Now, the arguments are exactly the same as in the case $Q=[0,1]^2$. This ends the proof of Theorem \ref{thm:mainbis}.

\begin{figure}[h!]
\centering
\includegraphics[scale=0.5]{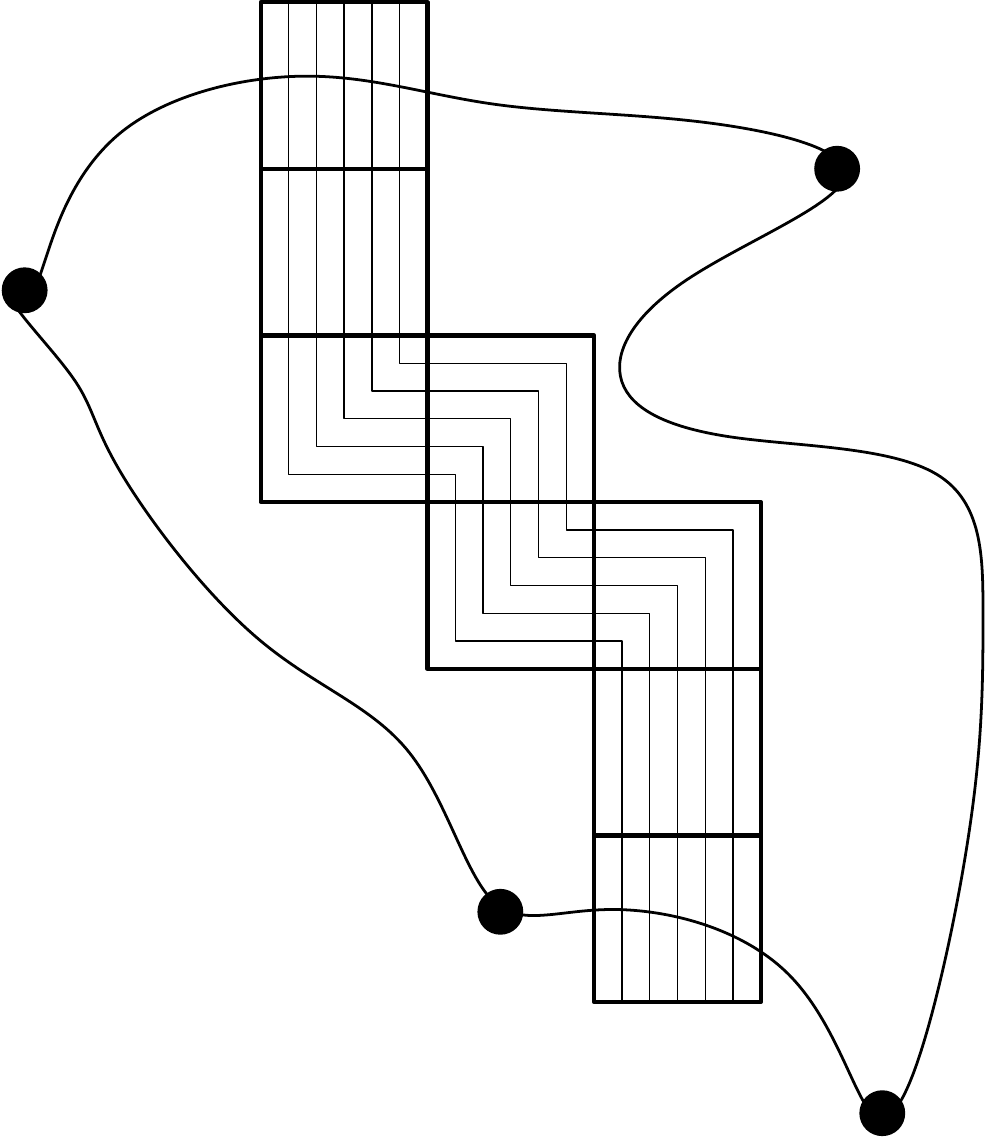}
\caption{The quad $nQ$, the squares $nS_1,\cdots,nS_m$, and the $\Theta(n^{1-\gamma})$ lines that replace the $\R \times \{k \}$'s.}\label{fig:algo}
\end{figure}

\section{Proof of Proposition \ref{pr.cor}.}\label{s.cor}

In this section, we prove Proposition \ref{pr.cor}, which is specific to the Bargmann-Fock field. However, it will be a consequence of the following more general result (and of Theorem \ref{th.main}).

\begin{proposition}\label{p.Wiener}
Let $q$ be a function that satisfies Condition \ref{cond:weak} and remember that $f(t)=q \star W_t$. For every event $A \in \mathcal{F}$, we have
\[
 \Var{\Pb{f(t) \in A | f(0) }} = \Var{\Pb{f(t) \in A | W_0 } }  =\Cov{ 1_{f(0) \in A} , 1_{f(2t) \in A} }  \, .
\]
\end{proposition}

\begin{proof}[Proof of Proposition \ref{pr.cor} using Proposition \ref{p.Wiener}]
Consider a 3D Bargmann-Fock field $F$, let $P(t)=\R^2 \times \{t\}$ and let $f^{\mathrm{hor} }(t,\cdot)=F_{|P(t)}$. It is easy to check that for every $t \in \R$ the joint coupling $(f^{\mathrm{hor}}(0),f^{\mathrm{hor} }(t))$ can be realized as follows:
\begin{align*}\label{}
f^\mathrm{hor}(t) & = e^{-\frac {t^2} 2} f^\mathrm{hor}(0) + \sqrt{1 - e^{-t^2}}\, \widetilde{f}^\mathrm{hor} \, ,
\end{align*}
where $\widetilde{f}^\mathrm{hor}$ is an independent copy of $f^\mathrm{hor}(0)$. Together with Proposition \ref{p.Wiener} (and Remark \ref{rem:coupl}), this implies that
\[
\Var{\Pb{f^\mathrm{hor}(t) \in A | f^\mathrm{hor}(0) }} = \Cov{ 1_{f(0) \in A} , 1_{f(t^2) \in A} }  \, .
\]
In particular, Proposition \ref{pr.cor} is now a direct consequence of Theorem~\ref{th.main}.
\end{proof}

The rest of this section is devoted to the proof of Proposition \ref{p.Wiener}. As we shall explain below, it seems one cannot avoid a spectral proof here. Let us start by recalling the simpler case of Boolean functions. We use the notations from Subsection \ref{ss:Sch-St}. If $g : \{-1,1\}^n \to \{0,1\}$ is a Boolean function, then we have the following useful identity: $\Eb{g(\sigma(t)) \md \sigma(0)} = \sum \hat g(S) e^{-t|S|} \chi_S(\sigma(0))$, which leads to 
\[
\Var{\Eb{g(\sigma(t)) \md \sigma(0)}} = \sum_{S\neq 0} \hat g(S)^2 e^{-2t|S|} \, .
\]
If we combine the above with the facts recalled in Subsection \ref{ss:Sch-St}, we obtain that
\[
\Var{\Eb{g(\sigma(t)) \md \sigma(0)}} = \Cov{ g(\sigma(2t)),g(\sigma(0)) } \, ,
\]
which is the discrete analogue of Proposition \ref{p.Wiener}. These spectral identities show that proving noise sensitivity in terms of covariance implies the seemingly stronger fact that the whole knowledge of the initial condition $\sigma(0)$ almost says nothing on the event $\{g(\sigma(t))=1\}$. At this stage, as the proof of Theorem \ref{th.main} proceeds by approximation,  
\[
1_{f \in A} \approx g_{N}(\sigma)
\]
where $g_{N}$ is a Boolean function on $\Omega_{N}=\{-1,1\}^{\Theta(N^2)}$, it seems natural to conclude by approximation, using  that we probably have
\[
\Var{\Eb{g_{N}(\sigma(t)) \md \sigma(0)}} \to \Var{\Pb{f(t) \in A | f(0)}} \, .
\] 
But conditional expectations are in general not continuous functions of the conditioning. Because of this, we argue differently as below. Before writing the proof, let us explain why Theorem \ref{th.main} and Proposition \ref{pr.cor} seem counter-intuitive even from the Fourier point of view.
\begin{remark}\label{rem:counter_int_fourier}
Remember that $F$ is a 3D Bargann-Fock and let $q(x)=q(x_1,x_2,x_3)=(2/\pi)^{1/4}e^{-|x|^2}$. As in Section \ref{s.proof}, we can approximate $F$ by a field $F^N_r=q_r \star W^N$ where $W^N$ is now a 3D discrete white noise and $q_r$ is a truncation of $q$ (for some suitable $r \leq n$). Let $P(t) = \R^2 \times \{t\}$, let $Q$ be a quad, and let $g=g_{r,n,N} \, : \, \{-1,1\}^{\Theta(n^3N^3)} \rightarrow \{0,1\}$ be the Boolean function such that $g_{r,n,N}(\sigma) = 1_{(F^N_r)_{|P(0)} \in \cross_0(nQ)}$. Then we can show that
\[
\Cov{ 1_{ (F^N_r)_{|P(0)} \in \cross_0(nQ)}, 1_{ (F^N_r)_{|P(t)} \in \cross_0(nQ)}  } = \sum_{S \neq \emptyset} \widehat{g}(S)\widehat{g}(S_{t})
\]
where $S_t$ is $S$ translated by $(0,0,-t)$ and where the sum is on subsets of a 3D grid of size $\Theta(n^3N^3)$. It does not seem obvious at all at first sight that the above is small when $n$ is large. Indeed, it does not seem clear why the above does not behave like the following analogous 2D quantity: consider a planar Bargmann-Fock field, some $x_0 \in \R^2$, and let $h=h_{r,n,N} \, : \, \{-1,1\}^{\Theta(n^2N^2)} \rightarrow \{0,1\}$ be the Boolean function corresponding to the crossing of $nQ$ by the planar analogue $f^N_r$ of $F^N_r$. Then,
\[
\Cov{ 1_{f^N_r \in \cross_0(nQ)} , 1_{f^N_r(\cdot +x_0) \in \cross_0(nQ)} } = \sum_{S \neq \emptyset} \widehat{h}(S)\widehat{h}(S_{x_0})
\]
where $S_{x_0}$ is $S$ translated by $x_0$ and where the sum is on subsets of a 2D grid of size $\Theta(n^2N^2)$. The above does not go to zero; $1_{f^N_r \in \cross_0(nQ)}$ and $1_{f^N_r(\cdot +x_0) \in \cross_0(nQ)}$ are highly correlated!
\end{remark}

\begin{proof}[Proof of Proposition \ref{p.Wiener}]
First note that (for instance by Remark \ref{rem:coupl}) the distribution of $f(t)$ conditioned on $W_0$ is the same as the distribution of $f(t)$ conditioned on $f(0)$. This implies that $\Var{\Pb{f(t) \in A | f(0) }} = \Var{\Pb{f(t) \in A | W_0 } }$. Let us now prove that $\Var{\Pb{f(t) \in A | W_0 } }  =\Cov{ 1_{f(0) \in A} , 1_{f(2t) \in A} }$. In our present continuous setting, we may apply the exact same idea as in the discrete case by relying on the appropriate spectral identities. Here we shall use the fact that our events can be seen as functionals  in $L^2(\sigma(W(dx)))$ and use \textbf{Wiener Chaos expansion} in $L^2(\sigma(W(dx)))$ which is the good analogue of the discrete Fourier expansion.
We may thus write for every event $A \in \mathcal{F}$ (recall $f:=q \star W$):
\begin{align*}\label{}
1_{f \in A} & = \sum_{k=0}^\infty   \iint_{(\R^2)^k} h_{k}(z_1,\ldots,z_k) dW(z_1)\ldots dW(z_k)
\end{align*}
where each $h_{k}\in L^2((\R^2)^k)$ and 
\[
\Pb{f \in A} = \sum_{k=0}^\infty   \iint_{(\R^2)^k} h_{k}^2(z_1,\ldots,z_k) d z_1 \ldots d z_k \, .
\]
For background on such Wiener Chaos expansions and multiple integrals with respect to Gaussian white noise, we refer to \cite{Janson,Peccati}, see also in the 2D setting the useful Section 2 in \cite{CSZ}. Now the following two identities will conclude the proof of Proposition \ref{p.Wiener}: 

\bnum
\item 
\[
\Cov{1_{f(0) \in A }, 1_{f(t)\in A}}
 = \sum_{k=1}^\infty e^{-t k} \iint_{(\R^2)^k} h_{k}^2(z_1,\ldots,z_k) d z_1 \ldots d z_k \, ;
\]
\item 
\[
\Var{\Pb{f(t)\in A | W_0}} = \sum_{k=1}^\infty e^{-2 t k} \iint_{(\R^2)^k} h_{k}^2(z_1,\ldots,z_k) d z_1 \ldots d z_k \, .
\]
\enum 

Even though these identities are classical, let us say a few words on their proofs.
(a) First, let us write $H_{k}(t)=\iint_{(\R^2)^k} h_{k}(z_1,\ldots,z_k) dW_t(z_1)\ldots dW_t(z_k)$ and let us note that the identities
\[
\Cov{H_{k}(0),H_k(t)}=e^{-tk}\iint_{(\R^2)^k} h_{k}^2(z_1,\ldots,z_k) d z_1 \ldots d z_k \, ;
\]
\[
\Var{\E \left[ H_k(t) \md W(0) \right]}=e^{-2t k}\iint_{(\R^2)^k} h_{n,k}^2(z_1,\ldots,z_k) d z_1 \ldots d z_k
\]
are direct consequences of the definition of multiple integrals with respect to the white noise in the case of special simple functions (and of density arguments), see for instance Subsection 2 of \cite{CSZ}.
(b) Second, let us note that a direct use of Fubini's theorem would not be sufficient here to exchange sum and expectation. One way to proceed is as follows. If $X,Y$ are any two variables in $L^2(\sigma(W(dx)))$, their cut-off Wiener chaos expansions (below order $n$, say) $X_n,Y_n$ are such that $X_n \overset{L^2}\longrightarrow X$ and $Y_n \overset{L^2}\longrightarrow Y$. Now, simply by linearity of the expectation, we have that $\Eb{X_n Y_n}$ is the finite sum over $k\in \{0,\ldots, n\}$ of the corresponding $k$-fold integrals. Using the fact that $X_nY_n \to XY$ in $L^1$ and therefore that $\Eb{X_nY_n} \to \Eb{XY}$, this justifies why one can interchange sum and expectation. 
\end{proof}

\section{Noise sensitivity implies sharp threshold}\label{s:PT}

In this section, we explain how one can combine our noise sensitivity result together with the analysis in Section \ref{s.cor} to prove results about the phase transition. In particular, we prove Theorem \ref{thm:pol}. We begin with the following result.
\begin{proposition}\label{p:dyna}
Let $q$ be a function that satisfies Condition \ref{cond:weak}. Also, let $(A_n)_{n \in \N}$ be a sequence of events in $\mathcal{F}$ and let $(t_n)_{n \in \N}$ be a sequence of positive numbers. If
\[
\Cov{ 1_{f(0) \in A_n} ,1_{f(t_n) \in A_n} } \underset{n \rightarrow +\infty}{\longrightarrow} 0
\]
and if $\Pro \left[ f \in A_n \right]$ is bounded away from $0$ and $1$, then for every $(s_n)_{n \in \N}$ sequence of positive numbers that satisfies $t_n = o(s_n)$, the number of times that $1_{f(t) \in A}$ switches from $0$ to $1$ between times $0$ and $s_n$ goes to $+\infty$ in probability as $n \rightarrow +\infty$.
\end{proposition}
\begin{proof}
The proof is exactly the same as the analogous result from \cite{BKS} for dynamical Bernoulli percolation i.e. this is a rather direct consequence of Proposition \ref{p.Wiener} and of the Markov property of the dynamics $t \mapsto f(t)$. We refer to Section 8 of Chapter I of \cite{book} for the proof.
\end{proof}


\begin{proof}[Proof of Theorem \ref{thm:pol}]
For every $\alpha >0$, let
\[
M(\alpha) = \sup_{x \in nQ,\, t\in[0,n^{-\alpha}]} |f(t,x)-f(0,x)|.
\]
By Proposition \ref{p:dyna} and Theorem \ref{thm:mainbis}, there exists $\alpha > 0$ such that, with high probability, there exists $t \in [0,n^{-\alpha}]$ such that $f(t) \in \cross_0(nQ)$. Note that, since the crossing events are increasing, this implies that with high probability $f(0) \in \cross_{M(\alpha)}(nQ)$. Hence, it is sufficient to show that with high probability, $M(\alpha)$ is polynomially small in $n$. By Lemma~\ref{lem:goalapp} (applied to $a=1/2$) and by an union bound on the $1 \times 1$ boxes of the grid $\Z^2$ that intersect $nQ$, we have the following, which completes the proof
\[
\Pro \left[ M(\alpha) \geq O(1) n^{-\alpha/4} \right] \leq O(1)n^2\exp(-\Omega(1)n^{-\alpha/2}) \, .
\]
\end{proof}

Let us note that it is shown in \cite{MV} (Section 6) that, if one assumes that $q$ satisfies Conditions \ref{cond:sym} and \ref{cond:pol} for some $\beta > 2$ and if
\[
\forall \ell > 0 \, , \Pro \left[ f \in \cross_\ell(nQ) \right] \underset{n \rightarrow +\infty}{\longrightarrow} 1
\]
for the quad $Q = [0,2] \times [0,1]$ (whose distinguished sides are the left and right sides), then we even have that the convergence is \textbf{exponentially fast}. As a result, Theorem \ref{thm:pol} gives a randomized algorithm proof of the recent sharpness result from \cite{R} which we have stated in Subsection \ref{ss:gen} (in order to prove the results in the case $\ell < 0$, one can use the results at $\ell > 0$ and the duality property of the model, see for instance Lemma A.9 of \cite{RVa} or Section 3.2 of \cite{MV}).







\appendix

\section{An estimate for a dynamical planar Gaussian field}\label{s.appendix}

Let $(W_t(dx))_{t \geq 0}$ be the dynamical planar white noise defined in Subsection \ref{ss:gen} and let $f(t,\cdot) = q \star W_t$ for some $q \, : \, \R^2 \rightarrow \R$ that satisfies Condition \ref{cond:weak}. So, there exist $c,\delta>0$ such that, for every $x \in \R^2$ and every multi-index $\alpha$ with $|\alpha| \leq 2$, we have
\begin{equation}\label{e:app:q}
| \partial^\alpha q(x) | \leq c(1 \vee |x|)^{-(1+\delta)} \, .
\end{equation}
Below, the constants in the $\Omega(1)$ and $O(1)$'s only depend on $c$ and $\delta$ except if we add some subscript. We can (and do) consider a modification of $(t,x) \mapsto f(t,x)$ that is continuous in $t$ and $x$ (see the application of Kolmogorov's continuity theorem at the end of this appendix). The goal of this appendix is to prove the following lemma.
\begin{lemma}\label{lem:goalapp}
Let $h \in [0,1]$ and $a \in ]0,1[$. Then,
\[
\Pro \left[ \sup_{t \in [0,h], x \in [0,1]^2} |f(t,x)-f(0,x)| \geq O_a(1) h^{(1-a)/2} \right] \leq O_a(1)\exp(-\Omega_a(1)h^{-a}) \, .
\]
\end{lemma}
\begin{remark}
Note that the lemma is sharp is the sense that
\[
\Pro \left[ \sup_{t \in [0,h], x \in [0,1]^2} |f(t,x)-f(0,x)| \geq h^{(1+a)/2} \right]
\]
goes to $1$ as $h$ goes to $0$, for any $a>0$. Actually, this is even true pointwise: for every $x$, the typical order of $|f(t,x)-f(0,x)|$ is $\sqrt{t}$.
\end{remark}


To prove Lemma \ref{lem:goalapp}, we use the two following classical results that we state for an a.s. continuous centered Gaussian process $(X(t,x))_{t \in [0,h], x \in [0,1]^2}$ on $[0,h] \times [0,1]^2$. Let $M:=\sup_{t \in [0,h],x \in [0,1]^2} X(t,x)$.
\begin{theorem}[Dudley's theorem]
Consider the following pseudo-metric on $[0,h] \times [0,1]^2$
\[
d((t,x),(s,y))=\sqrt{\E \left[ (X(t,x)-X(s,y))^2 \right]} \, .
\]
Also, let $\Delta$ be the $d$-diameter of $[0,h] \times [0,1]^2$ and, for every $\varepsilon>0$, let $N(d,\varepsilon)$ denote the minimal number of $d$-balls of radius $\varepsilon$ required to cover $[0,h] \times [0,1]^2$. Then,
\[
\E \left[ M \right] \leq 24\int_0^\Delta \sqrt{\log(N(d,\varepsilon))} d\varepsilon \, .
\]
\end{theorem}
\begin{theorem}[Borell-Tsirelson-Ibragimov-Sudakov (BTIS) inequality; Theorem 2.9 of \cite{AW}]
We have
\[
\Pro \left[ M \geq u+\E \left[ M \right] \right] \leq 2 \exp \left( -\frac{u^2}{\sup_{t \in [0,h],x \in [0,1]^2}\E \left[ X(t,x)^2 \right]} \right) \, .
\]
\end{theorem}

\begin{proof}[Proof of Lemma \ref{lem:goalapp}]
We will apply the above theorems to the field $X(t,x)=f(t,x)-f(0,x)$. First note that, by the BTIS inequality applied to $u=h^{(1-a)/2}$, it is sufficient - in order to prove the lemma - to show that (a) $\E \left[M \right] \leq O(\sqrt{h})$ and (b) $\E \left[ (f(t,x)-f(0,x))^2 \right] \leq O(h)$. In order to prove (a) and (b), we are going to show that, for every $s,t \in [0,h]$ and $x,y \in [0,1]^2$,
\begin{multline}\label{eq:var1}
\E \left[ \Big( f(t,x)-f(0,x)-(f(s,y)-f(0,y)) \Big)^2 \right]\\
\leq O(1) \left( |t-s|+h|x-y|^2 \right) \, .
\end{multline}
Let us first explain why \eqref{eq:var1} implies the estimates (a) and (b). This implies readily (b). The fact that this also implies (a) is a consequence of Dudley's theorem. Indeed, first note that \eqref{eq:var1} implies that
\[
\Delta \leq O(\sqrt{h}) \, .
\]
Let $\varepsilon \in ]0,\Delta]$ and note that (by \eqref{eq:var1}), for every $\varepsilon \leq \sqrt{h}$, the set $[0,\varepsilon^2]\times[0,\varepsilon/\sqrt{h}]^2 \subseteq [0,h]\times[0,1]^2$ is included in a $d$-ball of radius $\leq O(1)\varepsilon$. Moreover, one can cover $[0,h]\times[0,1]^2$ by using $\leq O(1)h^2/\varepsilon^4$ sets obtained by translating $[0,\varepsilon^2]\times[0,\varepsilon/\sqrt{h}]^2$. As a result,
\[
N(d,\varepsilon) \leq O(1)h^2/\varepsilon^4 \, ,
\]
and thus, by Dudley's theorem,
\[
\E \left[ M \right] \leq 24 \int_0^{O(\sqrt{h})} \sqrt{\log(O(h^2/\varepsilon^4))} d\varepsilon
\leq O(\sqrt{h}) \, .
\]
This ends the proof of Lemma \ref{lem:goalapp} provided that we prove \eqref{eq:var1}. Let us end this appendix by showing this variance estimate.
Let $x,y \in [0,1]^2$ and $s,t \in [0,h]$. We have
\begin{align*}
& \hspace{-0.4cm} \E \left[ \Big( f(t,x)-f(0,x)-(f(s,y)-f(0,y)) \Big)^2 \right]\\
 = & \E \left[ ( f(t,x)-f(0,x) )^2 \right] + \E \left[ ( f(s,y)-f(0,y))^2 \right]\\
& \hspace{2cm} -2 \E \left[ ( f(t,x)-f(0,x)) (f(s,y)-f(0,y)) \right]
\end{align*}
\begin{align*}
 = &(2-2e^{-t})\int q^2 + (2-2e^{-s})\int q^2\\ 
& \hspace{2cm}  - 2 (e^{-|t-s|}+1-e^{-s}-e^{-t}) \int q(x-z)q(y-z)dz\\
= & (4-2e^{-t}-2e^{-s})\int q^2 - (2e^{-|t-s|}+2-2e^{-s}-2e^{-t}) \int q(z)q(z+y-x)dz\\
 = & (4-2e^{-t}-2e^{-s})\int q(z)(q(z)-q(z+y-x))dz\\
& \hspace{2cm} +(2-2e^{-|t-s|})\int q(z)q(z+y-x)dz \, .
\end{align*}
Let us note that
\[
0 \leq 4-2e^{-t}-2e^{-s} \leq  O(h) \; ; \: \: 0 \leq 2-2e^{-|t-s|} \leq O(|t-s|)
\]
and
\[
\left| \int q(z)q(z+y-x)dz \right| \leq O(1) \, .
\]
Let us estimate $\int q(z)(q(z)-q(z+y-x))dz$. By Taylor's theorem and \eqref{e:app:q}, for every $z \in \R^2$ and every multi-index $\alpha$ such that $|\alpha| = 2$, there exists a function $R_\alpha^z \, : \, \R^2 \rightarrow \R$ such that $||R_\alpha^z||_\infty \leq O(|z \vee 1|^{-(1+\delta)})$ and $q(z) -q(z+y-x)$ equals
\[
\partial^{(1,0)}q(z)(x_1-y_1) + \partial^{(0,1)}q(z)(x_2-y_2) + \sum_{\alpha \in \N^2: \atop \alpha_1+\alpha_2=2} R_\alpha^z(x-y) \prod_{i=1}^2 (x_i-y_i)^{\alpha_i} \, .
\]
By using that $\int q\partial^{(1,0)}q = \int q \partial^{(0,1)}q = 0$ (by integration by parts), we obtain that
\[
\left| \int q(z) \left( q(z) -q(z+y-x) \right) dz \right| \leq O(1) \sum_{\alpha \in \N^2: \atop \alpha_1+\alpha_2=2} \prod_{i=1}^2 (x_i-y_i)^{\alpha_i} \leq O(|x-y|^2) \, .
\]
This implies the desired estimate \eqref{eq:var1} and completes the proof.
\end{proof}

Let us conclude this appendix by explaining why there is a modification of $(t,x) \mapsto f(t,x)$ that is a.s. continuous. A similar computation for the variance gives that, for every $t,s \in [0,1]$ and $x,y \in [0,1]^2$,
\[
\E \left[ ( f(t,x)-f(s,y) )^2 \right]  \leq O\left( |t-s|+|x-y|^2 \right) \, ,
\]
so for every $p>0$ we have
\[
\E \left[ ( f(t,x)-f(s,y) )^p \right] \leq O_p(1) \left( (|t-s|+|x-y|^2)^{p/2} \right) \, .
\]
In particular, if we apply Kolmogorov's continuity theorem to some $p>2$ (and if we use transitivity), we obtain that there exists a modification of $(t,x) \mapsto f(t,x)$ that is a.s. continuous.

\section{Scaling limit for the discrete white noise}\label{s:sobo}

As in Section \ref{s.proof}, for every integer $N \geq 1$ we consider the following discrete white noise on $\R^2$
\begin{align*}\label{}
W^N(x) := N \sum_{v\in \frac 1 N \Z^2} \sigma_v 1_{x \in v+[-1/N,1/N]^2} \,,
\end{align*}
where the random variables $\sigma_v$ are independent and $\Pro \left[ \sigma_v = 1\right] = \Pro \left[ \sigma_v = -1 \right]=1/2$. Moreover, we let each Bernoulli variable $\sigma_v$ evolve according to a rate 1 Poisson Point process, which induces a dynamics $t \mapsto W^N_t$.

\begin{lemma}\label{lem:law_sobo}
For any $t \geq 0$, any (open) square $S$ and any $\varepsilon>0$, we have the following convergence in law in $H^{-1-\varepsilon}(S) \times H^{-1-\varepsilon}(S)$:
\begin{align}
(W^N_0, W^N_t) \overset{\text{law}}\longrightarrow (W_0,W_t)\,,
\end{align}
where $W_t = e^{-t}W_0+ \sqrt{1- e^{-2t}} \widetilde{W}$ and $\widetilde{W}$ is a white noise independent of $W$.
\end{lemma}

\begin{proof}
We will use the following: i) the eigenfunctions $\Psi_n$ of the Laplacian on $S$ (with eigenvalues $\lambda_n = \Theta(n)$) are uniformly bounded and ii) for every $u \in H^s(S), \, ||u||_{H^s(S)}^2=\sum_n <u,\Psi_n>^2 \lambda^s_n$.

Let us first prove tightness. By Rellich theorem, any closed ball of $H^{-1-\varepsilon/2}(S)$ is a compact subset of $H^{-1-\varepsilon}(S)$. As a result, tightness follows from the fact that, for every $\varepsilon>0$, $\sup_{N \geq 1} \E \left[ || W_0^N ||_{H^{-1-\varepsilon}(S)}^2 \right] < +\infty$, which follows from an elementary computation by using i) and ii) above.

We now prove that the only possible limit is $(W_0,W_t)$. To this purpose, we note that, by ii), the law of a random variable $h$ in $H^s(S)$ is determined by the joint law of the variables $<h,\Psi_n>$. The result now follows for instance from an explicit computation of the limit of the Lévy transform of $(<W_0^N,\psi_1>,<W_t^N,\psi_1>,\cdots,<W_0^N,\psi_m>,<W_t^N,\psi_m>)$, for every $m$.
\end{proof}

\section{A regularity result}\label{s.reg}

The following is a direct consequence\footnote{In this lemma, one needs the additional hypothesis that ``$f$ is not degenerate'', which is actually a consequence of the fact that the support of the Fourier transform of $q$ contains an open set, see Theorem 6.8 of \cite{Wen}.} of Lemma A.9 of \cite{RVa}: Let $L \subset \R^2$ be a line. Under Condition \ref{cond:sym}, a.s. the following holds: i) The sets $\{ f \geq 0\}$ and $\{ f \leq 0 \}$ are two $\mathcal{C}^1$-smooth $2$-dimensionnal manifolds with boundary, ii) $\partial \{ f \geq 0 \} = \partial \{ f \leq 0 \} = \{ f = 0 \}$ and iii) the set $\{ f = 0 \}$ is a $\mathcal{C}^1$-smooth $1$-dimensionnal manifold that intersects $L$ transversally.

This implies in particular that, for every quad $Q$, a.s. there exists some (random) $\varepsilon>0$ such that, for all continuous function $g$ that satisfies $\parallel f-g \parallel_{Q,\infty} \leq \varepsilon$, we have $f \in \cross_0(Q)$ if and only if $g \in \cross_0(Q)$.


\bibliographystyle{alpha}

\begin{thebibliography}
{GPS13b} 

\bibitem[AB18]{AB}
Ahlberg, D. and Baldasso, R. {\em Noise sensitivity and Voronoi percolation.} Electronic Journal of Probability, 23, 2018.

\bibitem[ABGM14]{ABGM}
Ahlberg, D., Broman, E., Griffiths, S. and Morris, R. {\em Noise sensitivity in continuum percolation.} Israel Journal of Mathematics, 201(2), 847-899, 2014.

\bibitem[AGMT16]{AGMT}
Ahlberg, D., Griffiths, S., Morris, R. and Tassion, V. {\em Quenched Voronoi percolation.} Advances in Mathematics, 286, 889-911, 2016.

\bibitem[Ale96]{Alex}
Alexander, K.S. {\em Boundedness of level lines for two-dimensional random fields.} The Annals of Probability 24.4, 1653--1674, 1996.

\bibitem[AW09]{AW} Azaïs, J.-M. and Wschebor, M.
\newblock {\em Level sets and extrema of random processes and fields.}
\newblock John Wiley $\&$ Sons, 2009.

\bibitem[BDS06]{BDS}
Bogomolny, E., Dubertrand, R. and Schmit, C. {\em SLE description of the nodal lines of random wavefunctions.} Journal of Physics A: Mathematical and Theoretical, 40(3), 381, 2006.

\bibitem[BG17a]{BG}
Beffara, V. and Gayet, D. {\em Percolation of random nodal lines.} Publications math\'ematiques de l'IH\'ES 126.1: 131--176, 2017.

\bibitem[BG17b]{BGbis}
Beffara, V. and Gayet, D. {\em Percolation without FKG.} arXiv preprint arXiv:1710.10644, 2017.


\bibitem[BM18]{BM}
Beliaev, D. and Muirhead, S. {\em Discretisation schemes for level sets of planar Gaussian fields.} Communications in Mathematical Physics 359.3 : 869--913, 2018.

\bibitem[BMR18]{BMR}
Beliaev, D., Muirhead, S. and Rivera, A. {\em A covariance formula for topological events of smooth Gaussian fields.} arXiv preprint arXiv:1811.08169, 2018.

\bibitem[BKS99]{BKS}
Benjamini, I., Kalai, G. and Schramm, O. {\em Noise sensitivity of Boolean functions and applications to percolation.} Publications Mathématiques de l'Institut des Hautes Etudes Scientifiques 90.1: 5--43, 1999.

\bibitem[BS02]{BSa}
Bogomolny, E. and Schmit, C. {\em Percolation model for nodal domains of chaotic wave functions.} Physical Review Letters, 88(11), 114102, 2002.

\bibitem[BS07]{BSb}
Bogomolny, E. and Schmit, C. {\em Random wavefunctions and percolation.} Journal of Physics A: Mathematical and Theoretical, 40(47), 14033, 2007.

\bibitem[BS98]{BS}
Benjamini, I. and Schramm, O. {\em Exceptional planes of percolation.} Probability theory and related fields 111.4: 551--564, 1998.

\bibitem[BR06]{BR}
Bollobás, B. and Riordan, O. {\em The critical probability for random Voronoi percolation in the plane is 1/2.} Probability theory and related fields 136.3: 417--468, 2006.

\bibitem[CSZ16]{CSZ}
Caravenna, F., Sun, R. and Zygouras, N. {\em Polynomial chaos and scaling limits of disordered systems.} Journal of the European Mathematical Society, 19(1), pp.1--65. 2016. 

\bibitem[DRT19]{OSSS}
Duminil-Copin, H., Raoufi A. and Tassion, V. {\em Sharp phase transition for the random-cluster and Potts models via decision trees.} Annals of Mathematics 189.1: 75--99, 2019.

\bibitem[GPS10]{GPS}
Garban, C., Pete, G. and Schramm, O. {\em The Fourier spectrum of critical percolation.} 
 Acta Mathematica 205.1: 19--104, 2010.

\bibitem[GS14]{book}
Garban, C. and Steif, J.E.
\newblock {\em Noise sensitivity of Boolean functions and percolation}.
\newblock Cambridge University Press, 2014.

\bibitem[HPS97]{HPS}
H\"aggstr\"om, O., Peres, Y. and Steif, J.E. {\em Dynamical percolation.} Annales de l'Institut Henri Poincare (B) Probability and Statistics. Vol. 33. No. 4, 1997.

\bibitem[Jan97]{Janson}
Janson, S.
\newblock{Gaussian Hilbert spaces}.
\newblock Vol. 129. Cambridge university press, 1997.

\bibitem[KMS12]{KMS}
Keller, N., Mossel, E. and Sen, A. {\em Geometric influences.} The Annals of Probability 40(3), 1135-1166, 2012.

\bibitem[LSW02]{LSW}
Lawler, G., Schramm, O. and Werner, W. {\em One-arm exponent for critical 2D percolation.} Electronic Journal of Probability, 7, 2002.

\bibitem[MV18]{MV}
Muirhead, S. and Vanneuville, H. {\em The sharp phase transition for level set percolation of smooth planar Gaussian fields.} Annales de l'Institut Henri Poincare (B) Probability and Statistics. Vol. 56. No. 2, 2020.

\bibitem[Pit12]{Pit}
Piterbarg, V.I.
\newblock {\em Asymptotic methods in the theory of Gaussian processes and fields}.
\newblock American Mathematical Soc., 2012.

\bibitem[PT10]{Peccati}
Peccati, G. and Taqqu, M.S.
\newblock{ \em Wiener Chaos: Moments, Cumulants and Diagrams}.
\newblock Springer-Verlag, 2010.

\bibitem[RV17a]{RVa}
Rivera, A. and Vanneuville, H. {\em Quasi-independence for nodal lines.} Annales de l'Institut Henri Poincare (B) Probability and Statistics. Vol. 55. No. 3, 2019.

\bibitem[RV17b]{RVb} Rivera A. and Vanneuville, H. {\em The critical threshold for Bargmann-Fock percolation.} Annales Henri Lebesgue, Vol. 3, 2020.

\bibitem[Riv19]{R} Rivera, A. {\em Talagrand's inequality in planar Gaussian field percolation.} arXiv peprint, arXiv:1905.13317, 2019.

\bibitem[SS10]{SS}
Schramm, O. and Steif, J.E.,  {\em Quantitative noise sensitivity and exceptional times for percolation.} Annals of Mathematics. 171.2: 619--672, 2010.

\bibitem[SW01]{SW}
Smirnov, S. and Werner, W. {\em Critical exponents for two-
dimensional percolation.} Math. Res. Lett., 8(5-6):729-744, 2001.

\bibitem[Wei84]{W}
Weinrib, A. {\em Long-range correlated percolation.} Physical Review B, 29(1), 387, 1984.

\bibitem[Wen05]{Wen}
Wendland, H. {\em Scattered Data Approximation.} Cambridge University Press, 2005.

\end{thebibliography}

\end{document}